\documentclass[sn-chicago,Numbered,iicol]{sn-jnl}

\usepackage{tikz,pgfplots}
\usetikzlibrary{shapes,arrows,positioning}
\usetikzlibrary{shapes.geometric, arrows, shadows}
\usetikzlibrary{fit,backgrounds}
\usepackage{graphicx}%
\usepackage{multirow}
\usepackage{multicol}
\usepackage{amsmath,amssymb,amsfonts}%
\usepackage{amsthm}%
\usepackage{longtable}
\usepackage{geometry}
\usepackage{mathrsfs}%
\usepackage[title]{appendix}%
\usepackage[utf8]{inputenc}
\usepackage{xcolor}%
\usepackage{textcomp}%
\usepackage{manyfoot}%
\usepackage{booktabs}%
\usepackage{algorithm}%
\usepackage{algorithmicx}%
\usepackage{hhline}
\usepackage{algpseudocode}%
\usepackage{listings}%

\hypersetup{
    colorlinks=true,
    linkcolor=black,
    filecolor=magenta,      
    urlcolor=[rgb]{0.3,0.65,0.93},
    pdftitle={Overleaf Example},
    pdfpagemode=FullScreen,
    }
\makeatletter
\newcommand*\bigcdot{\mathpalette\bigcdot@{.6}}
\newcommand*\bigcdot@[2]{\mathbin{\vcenter{\hbox{\scalebox{#2}{$\m@th#1\bullet$}}}}}
\newcommand{\X}{\mathbb{X}}
\newcommand{\K}{\mathcal{K}}

\def\NN{\mbox{I\hspace{-.1em}N}}
\def\R{\mbox{I\hspace{-.1em}R}}
\def\C{\mbox{I\hspace{-.55em}C}}

\newtheorem{remark}{Remark}%
\newtheorem{prop}{Proposition}
\newtheorem{defn}{Definition}

\begin{document}

\title[Article Title]{A numerical Koopman-based framework to estimate regions of attraction for general vector fields}

\author*[1]{\fnm{François-Grégoire} \sur{Bierwart}}\email{francois-gregoire.bierwart@unamur.be}

\author[1]{\fnm{Alexandre} \sur{Mauroy}}

\affil[1]{\orgdiv{Department of Mathematics} and \orgname{Namur Institute for Complex Systems (naXys), University of Namur}, \orgaddress{\country{Belgium}}}

\abstract{In this paper, we develop a comprehensive framework to estimate regions of attraction of equilibria for dynamics associated with general vector fields. This framework combines Koopman operator-based methods with rigorous validation techniques. A candidate Lyapunov function is constructed with approximated Koopman eigenfunctions and further validated through polynomial approximation, either with SOS-based techniques or with a worst-case approach using an adaptive grid. The framework is general, not only since it is adapted to non-polynomial vector fields, but also since the Koopman operator can be approximated with general bases yielding non-polynomial Lyapunov functions. The performance of the method is illustrated with several numerical examples.}


\maketitle

\section{Introduction}

It is undeniable that stability properties are crucial in systems analysis and control design. In particular, global stability---as opposed to local/linear stability---is directly related to the estimation of regions of attraction (ROA), a problem which plays a key role in safety-critical applications. Although global stability is a basic property of dynamical systems, yet estimating the ROA of an equilibrium point is a challenging task. The main (and almost unique) technique to do so relies on the existence of a Lyapunov function, i.e. a positive function that decreases along every trajectory generated by the dynamics \cite{KhalilNonlinearControl}. Lyapunov function design is generally not systematic and the validation of the obtained function, when required, can be difficult. These issues are particularly limiting when the vector fields describing the dynamics are not polynomial, a situation which is prevalent in most real-life applications.  

The most popular systematic technique to construct valid Lyapunov functions relies on sum-of-squares (SOS) polynomials, which are non-negative polynomials that can be used to encode specific Lyapunov inequalities \cite{parrilo2003semidefinite}. Through this framework, the Lyapunov function design problem can be recast as an efficient semi-definite program, and is currently implemented in off-the-shell toolboxes (see e.g. \cite{papachristodoulou2013sostools}). Obviously, SOS-based stability methods are primarily tailored to polynomial vector fields, but they have been extended to non-polynomial vector fields at the cost of limited applicability \cite{papachristodoulou2005analysis, chesi2009estimating} or conservativeness due to approximation bounds \cite{bauml2009stabilization, wu2014domain}. Besides the SOS method, different approaches have been developed, such as vector field interpolation (\cite{saleme2011estimation}), solution to the Zubov's equation \cite{vannelli1985maximal,meng2024koopman}, construction of piecewise linear Lyapunov functions through linear programming (\cite{julian1999parametrization}), and collocation kernel-based methods (\cite{giesl2007meshless,giesl2016approximation}), to list a few. We refer to \cite{giesl2015review} for a general overview of existing methods for constructing Lyapunov function.

In recent years, neural network based methods have gained increasing interest in the context of stability, allowing to learn Lyapunov functions either for polynomial vector fields or general vector fields \cite{chang2019neural,ravanbakhsh2019learning}. In this context, the validation step usually relies on SMT solvers (e.g. Z3, dReal \cite{gao2013dreal}) which might not be adapted to any type of activation functions and can be slow to converge (e.g. during the counter-example process). 

More recently, global stability analysis and Lyapunov function design have been investigated in the Koopman operator framework (see e.g. \cite{budivsic2012applied} for a review). The Koopman (or composition) operator describes the evolution of observable-functions evaluated along the system trajectories \cite{koopman1931hamiltonian}. This framework provides a linear description of nonlinear dynamics in the so-called lifted state space, a description which is amenable to stability analysis. In particular, it is shown in \cite{mauroy2014global} that the eigenfunctions of the Koopman operator can be used to assess global stability in a systematic way. And alternatively, a linear Lyapunov equation expressed in the lifted state variables can be solved efficiently to construct Lyapunov functions \cite[Chapter 3]{mauroy2020koopman}. Importantly, these techniques do not require specific assumptions on the form of the vector field, which does not need to be polynomial.  Moreover, they allow to construct a wide range of Lyapunov functions such as polynomials, trigonometric, or exponential functions.

While the Koopman operator leads to a powerful framework to obtain Lyapunov functions in a systematic way, this comes with a cost. Since the operator is infinite-dimensional, it is typically approximated in a finite-dimensional subspace spanned by a predefined set of basis functions. It follows that the functions obtained with those approximations are only candidate Lyapunov functions, and should therefore be validated. This validation step is particularly crucial in the context of ROA estimation, but has not been performed in a comprehensive and rigorous way in the original works \cite{mauroy2014global,mauroy2020koopman}. In this context, there is a need for combining Koopman operator methods allowing to obtain Lyapunov functions candidates through efficient linear algebraic techniques, with powerful validation methods such as SOS techniques. This has been attempted recently in a data-driven context, where validation techniques (e.g. SOS methods, neural networks) are based on the assumption that snapshot trajectories are available and can possibly be generated at will \cite{bramburger2023auxiliary,deka2022koopman,Mauroy2023data}.

In this paper, we push further the results of \cite{mauroy2014global,mauroy2020koopman} by developing a comprehensive numerical framework for ROA estimation of equilibria in the case of general dynamical systems. We consider a classical non data-driven setting, where we forbid the use of simulated trajectories, but assume a perfect knowledge of the possibly non-polynomial vector field. We tackle the problem by combining a Koopman-based Lyapunov design method with rigorous validation techniques. The main validation method leverages SOS-based techniques, which are solely used to validate the Lyapunov function and estimate the ROA, while the Lyapunov function candidate is obtained in a prior step through the Koopman operator framework. Another proof-of-concept validation method is proposed and based on a ``worst-case" approach combined with an adaptive grid. Our whole framework is adapted to non-polynomial vector fields through the use of polynomial approximations (Taylor approximation and minimax approximation). Moreover, it also allows to use non-polynomial basis functions (e.g. Gaussian radial basis functions), which have the potential to provide better approximations of the Koopman operator, and therefore better (non-polynomial) candidate Lyapunov functions. Also, with appropriate basis functions, our approach can potentially be used to design Lyapunov function candidates for high-dimensional systems. However, the proposed validation methods suffer from the curse of dimensionality and can only be used with systems of dimension less than or equal to three, which will be referred to as low-dimensional systems\footnote{Systems of dimension greater or equal to 10 will be referred to as high-dimensional systems throughout the paper.} Our numerical ROA estimation method is implemented in the open-access (Matlab) toolbox KOSTA, which is available at \url{https://github.com/FgBierwart/KOSTA-Toolbox}.  

The rest of the paper is organized as follows. In Section \ref{preliminaries}, we introduce the Koopman operator framework in the context of stability analysis. In Section \ref{validation-section}, we present the problem of Lyapunov function validation through SOS programming and through an adaptive grid based on a worst-case approach. We also describe how an inner approximation of the region of attraction is computed. Numerical examples are presented in Section \ref{Numerical-section}, for both polynomial and non-polynomial vector fields. Concluding remarks and perspectives are given in Section \ref{sec:conclu}.

\section{Koopman formalism for Lyapunov function design}\label{preliminaries}

In this section, we briefly introduce the problem of estimating regions of attraction of equilibria. We also describe the general methodology relying on Lyapunov functions computed with the Koopman operator framework.

\subsection{Regions of attraction} 

Consider the dynamical system
\begin{equation}\label{eq}
\dot{x} = F(x), \quad x\in\X
\end{equation}
where $\X\subset\R^n$ is a compact set and the vector field $F$ is Lipschitz continuous but otherwise general (e.g. non-polynomial). We denote by $\varphi^{t}(x) : \R^+ \times \mathbb{X} \rightarrow \mathbb{X}$ the flow map generated by (\ref{eq}). Moreover, we assume that the system (\ref{eq}) admits a locally stable equilibrium at the origin.

We will aim at computing an inner approximation of the region of attraction $\mathcal{R}$ of the equilibrium, defined by
$$
\mathcal{R}=\{x\in\X \mid \lim_{t\rightarrow\infty}\varphi^{t}(x) = 0\}.
$$
To do so, we will rely on classic stability theory \cite{KhalilNonlinearControl} and construct Lyapunov functions $V: \X \longrightarrow \X$ which satisfy $V(x)>0$ and $\dot{V}(x) = \nabla V(x)^{\top}F(x) < 0$ for all $x\in \mathcal{R} \setminus\{0\}$, and $V(0)=0$. Note that, in practice, we will obtain approximated Lyapunov functions that satisfy the condition $\dot{V}<0$ only over a subset $\mathcal{S} \subseteq \mathcal{R}$. We call this set $\mathcal{S} = \{x\in \X~|~ \dot{V}(x) < 0\}$ the \emph{validity region} of the Lyapunov function. Let $\Omega_{\gamma} = \{x\in \X \mid V(x)\leq \gamma\}$ be a sublevel set of $V$ for some $\gamma \geq 0$, and let $\partial\Omega_{\gamma} = \{x\in \X \mid V(x)= \gamma\}$ be the boundary of $\Omega_{\gamma}$. Then, an approximation of the ROA is given by the set $\Omega_{\gamma_2}$ if there exists $0\leq \gamma_1<\gamma_2$ such that $\Omega_{\gamma_2} \setminus \Omega_{\gamma_1} \subseteq \mathcal{S}$. In this case, any trajectory with an initial condition in $\Omega_{\gamma_2}$ converges to $\Omega_{\gamma_1}$. Note that $\Omega_{\gamma_1}=\{0\}$ if $\gamma_1=0$. Therefore we will aim at solving the following optimization problem:
\begin{equation}
    \label{eq:ROA}
\begin{array}{cl}
\displaystyle \max_{\gamma_1,\gamma_2\geq 0} & \gamma_2-\gamma_1\\
~~\textrm{s.t.} & \overline{\Omega_{\gamma_2} \setminus \Omega_{\gamma_1}} \subseteq \mathcal{S}\\
\end{array}
\end{equation}

\noindent where $\overline{\phantom{S}}$ denotes the closure of a set. 

More accurate approximations of the ROA can also be obtained by combining several Lyapunov functions. This is summarized in the following proposition.\\ 

\begin{prop}
\label{multiple-lyap}
Assume that $\{\gamma_{1}^{(i)}\}_{i=1}^m$ and $\{\gamma_{2}^{(i)}\}_{i=1}^m$ are feasible solutions to \eqref{eq:ROA} for a family of $m$ Lyapunov functions $\{V^{(i)}\}_{i=1}^m$. Then, if $\cup_{i=1}^m \Omega_{\gamma_{1}^{(i)}} \subset \cap_{i=1}^m\Omega_{\gamma_{2}^{(i)}}$, any trajectory starting from $\cup_{i=1}^m \Omega_{\gamma_{2}^{(i)}}$ converges to $\cap_{i=1}^m\Omega_{\gamma_{1}^{(i)}}$.
\end{prop}

\begin{proof}
Let $x\in\cup_{i=1}^m \Omega_{\gamma_{2}^{(i)}}$ and, without loss of generality, assume that $x\in\Omega_{\gamma_{2}^{(1)}}$. Thus, there exists $t_1\geq0$ such that $\varphi^{t_1}(x)\in\Omega_{\gamma_{1}^{(1)}}$. Indeed, assume by contradiction that $\varphi^t(x) \in \overline{\Omega_{\gamma_2^{(1)}} \setminus \Omega_{\gamma_1^{(1)}}}$ for all $t\geq0$. Moreover, the continuous function $\dot{V}^{(1)}$ attains its maximal value $\dot{V}^{(1)}_{max}$ on $\overline{\Omega_{\gamma_2^{(1)}} \setminus \Omega_{\gamma_1^{(1)}}}$, with $\dot{V}^{(1)}_{max}<0$ since $\overline{\Omega_{\gamma_2^{(1)}} \setminus \Omega_{\gamma_1^{(1)}}} \subseteq \mathcal{S}$. It follows that
$$
\begin{array}{rcl}\label{proof}
V^{(1)}(\varphi^t(x)) & = & V^{(1)}(x) + \int_0^t \dot{V}^{(1)}(\varphi^{\tau}(x))~\mathrm{d}\tau,\\[0.2cm]
& < & V^{(1)}(x) -\dot{V}^{(1)}_{max} t
\end{array}
$$
so that $V^{(1)}(\varphi^t(x))<\gamma_{1}^{(1)}$ for $t$ large enough, which is a contradiction.

Since $\cup_{i=1}^m \Omega_{\gamma_{1}^{(i)}} \subset \cap_{i=1}^m\Omega_{\gamma_{2}^{(i)}}$, we have that $\varphi^{t_1}(x)\in\Omega_{\gamma_{2}^{(2)}}$ and, by the same argument as above, there exists $t_2\geq0$ such that $\varphi^{t_2+t_1}(x)\subset\Omega_{\gamma_{1}^{(2)}}$. Since $\Omega_{\gamma_{1}^{(1)}}$ is an invariant set, we have $\varphi^{t_2+t_1}(x)\subset\Omega_{\gamma_{1}^{(1)}} \cap \Omega_{\gamma_{1}^{(2)}}$. Then, we can repeat this scheme for all $i$ and show that there exist $t_1,\dots,t_m\geq0$ such that $\varphi^{t_1+\dots+t_m}(x) \in \Omega_{\gamma_{1}^{(1)}} \cap \cdots \cap \Omega_{\gamma_{1}^{(m)}}$.      
\end{proof}

As we will see below, our numerical method based on the Koopman operator framework can typically provide a wide range of Lyapunov functions. According to the above result, this is useful to improve the estimation of the ROA (see Example 3 in Section \ref{non-poly VF}).

\subsection{Koopman operator framework for stability analysis}\label{Koopman-stability}

Lyapunov functions will be computed through the Koopman operator framework, which allows to leverage linear stability methods for general nonlinear systems.\\ 
\begin{defn}
Let $\mathcal{F}$ be a Banach space of observable functions. We define the Koopman semigroup of linear operators as the family $\{\K_t\}_{t\geq 0}$, $\K_t : \mathcal{F} \rightarrow \mathcal{F}$ such that $
\K_t f = f \circ \varphi^{t},~ \forall f \in \mathcal{F}.
$\\ 
\end{defn}

\noindent It is known that the semigroup of Koopman operators is strongly continuous in $C(\X)$ \cite{Engel_Nagel,mauroy2020koopman}, so that the infinitesimal generator
$$
\mathcal{L}f ~ = ~ \lim_{t \rightarrow 0^{+}}\dfrac{\mathcal{K}_tf - f}{t} ~=~ F\cdot\nabla f
$$
is well-defined for all $f\in C^1(\X)$. We note that, in contrast to the semigroup, the expression of the infinitesimal generator is directly related to the vector field $F$.

Since the Koopman operator is linear, we can consider its spectral properties, which capture the dynamical properties of the underlying nonlinear system \cite{mezic2005spectral}. 
The eigenfunctions $\phi_{\lambda}$ and associated eigenvalues $\lambda \in \C$ of the infinitesimal generator satisfy $\mathcal{L}\phi_{\lambda} = \lambda\phi_{\lambda}.
$
According to the spectral mapping theorem \cite{Engel_Nagel}, they also verify the equality $\mathcal{K}_t\phi_{\lambda} ~ = ~e^{\lambda t}\phi_{\lambda}$. As shown in \cite{mauroy2014global}, the eigenfunctions $\phi_{\lambda}$ associated with eigenvalues such that $\Re\{\lambda\}<0$ capture the stability properties of the dynamics. In particular, if the origin is a locally stable hyperbolic equilibrium, there are $n$ eigenfunctions, which are continuously differentiable in the neighborhood of the origin and associated with the eigenvalues $\lambda_i$ of the Jacobian matrix $JF(0)$ of $F$ computed at $0$. These eigenfunctions are continuous on the ROA $\mathcal{R}$ and allow to construct a generic Lyapunov function
\begin{equation}\label{lyap}
V(x) = \sum_{i=1}^m \alpha_i\left|\phi_{\lambda_i}(x)\right|^2,~~~\forall x\in\mathcal{R}
\end{equation}
with $\alpha_i>0$, $i=1,\dots,m$. Indeed, we can easily verify that $V(x)>0$ for all $x\in \mathcal{R}\setminus\{0\}$ (note that $V(0)=0$ since $\phi_{\lambda_i}(0)=0$) and
$$
\begin{array}{rcl}
\dot{V}(x) & = & \mathcal{L}V (x) = \sum_{i=1}^m \alpha_i  \mathcal{L} \left|\phi_{\lambda_i}(x)\right|^2\\[0.3cm]
& \leq & 2\max_{i \in\{1,\dots,n\}} \Re\{\lambda_i\} V(x) < 0,
\end{array}
$$
for all $ x\in \mathcal{R}\setminus\{0\}$, where we used 
$$
\begin{array}{rcl}
\mathcal{L}|\phi_{\lambda_i}(x)|^2   &=& \displaystyle\lim_{t\rightarrow 0^+} \frac{\mathcal{K}^t|\phi_{\lambda_i}(x)|^2-|\phi_{\lambda_i}(x)|^2}{t} \\[0.4cm]
&=& \displaystyle\lim_{t\rightarrow 0^+} \frac{(e^{2\Re\{\lambda_i\}t} - 1)|\phi_{\lambda_i}(x)|^2}{t}\\[0.6cm]
&=& 2\Re\{\lambda_i\}|\phi_{\lambda_i}(x)|^2.
\end{array}
$$
Note also that several Lyapunov functions can be obtained by tuning the values $\alpha_i$. By default, we will set $\alpha_i=1$ for all $i$.

\subsection{Systematic computation of Lyapunov functions} \label{method}

The Koopman operator is defined on an infinite-dimensional space, so that the computation of its eigenfunctions needed to obtain the Lyapunov function is not straightforward. A common way to circumvent this issue is to compute a finite-dimensional approximation of the Koopman operator (see e.g. \cite[Chapter 1]{mauroy2020koopman} for an overview). Let $\mathcal{F}_{N}$ be a finite-dimensional subspace of $\mathcal{F}$ spanned by a set of $N$ basis functions $\{\psi_j(x)\}_{j=1}^{N}$. A typical choice is the subspace of monomials $\mathcal{F}_P^d = \{x_1^{a_1}\ldots\,x_n^{a_n} \mid \sum_i a_i \leq d, a_i \in \NN \}$ of total degree less or equal to $d \in \NN$. The use of monomials is particularly well-suited to low-dimensional systems ($n\leq 3$), but other types of basis functions (e.g. Gaussian radial basis functions, see Section \ref{Numerical-section}) might be preferred in higher-dimensional settings. Note also that the choice of basis functions can potentially affect the quality of the results, even for low-dimensional systems. We also use a projection operator $\Pi : \mathcal{F}\rightarrow\mathcal{F}_{N}$ that maps any function $f \in \mathcal{F}$ onto the finite-dimensional space $\mathcal{F}_{N}$. An appropriate choice of projection operator depends on the type of vector field and basis functions. In particular, we observe that $\mathcal{L}\psi_i = \nabla \psi_i^{\top} F$ is polynomial if the vector field is polynomial and if the basis functions are monomials. In this case, it is convenient to use a finite section method associated with the truncation projection operator
    $$
\Pi(x_1^{a_1}\ldots\,x_n^{a_n}) = 
\begin{cases}
     x_1^{a_1}\ldots\,x_n^{a_n} & \text{if } \sum_i a_i \leq d,\\
     0 & \text{otherwise}. 
\end{cases}
$$
In other cases, we will rely on the standard orthogonal $L^2$ projection
$$
\Pi f = \operatorname*{argmin}_{g \in \mathcal{F}_N} \int_{\X_\Pi} \left|f-g\right|^2 dx
$$
with $\X_\Pi=[-0.1,0.1]^n$ by default. This choice is motivated by a better approximation of the eigenvalues of the Jacobian matrix at the origin, but may depend on the dynamics. In practice, the projection is obtained with a Galerkin method, where inner products are approximated through a Monte-Carlo technique. In the specific case of (Gaussian) radial basis functions and Monte-Carlo samples taken as the centers of these functions, the approximation boils down to a kernel-based method (see \cite{williams2014kernel}).

Next, an approximation of the infinitesimal generator $\mathcal{L}$ is given by 
$
\mathcal{L}_{N} ~:=~ \Pi\,\mathcal{L}\hspace{-0.12cm}\mid_{\mathcal{F}_{N}} ~:~ \mathcal{F}_{N} \rightarrow \mathcal{F}_{N},
$
where $\,\mathcal{L}\hspace{-0.12cm}\mid_{\mathcal{F}_{N}}$ is the restriction of $\mathcal{L}$ to $\mathcal{F}_{N}$. For an arbitrary function $f \in \mathcal{F}_{N}$, it holds that
$$
f(x) = \sum_{i=1}^{N}a_i\psi_i(x) := a^{\top}\Psi(x)
$$
and
$$(\mathcal{L}_{N}f)(x) = \sum_{i=1}^{N}b_i\psi_i(x) := b^{\top}\Psi(x)
$$
with $\Psi(x) = \left(\psi_1(x),\ldots,\psi_{N}(x)\right)^{\top}$, and where $a:=(a_i,\ldots,a_{N})^{\top}$ and $b:=(b_i,\ldots,b_{N})^{\top}$ are the coefficients of the expansion of $f$ and $\mathcal{L}_{N}f$, respectively, in the basis functions. It follows that we can define a matrix representation $\mathbf{L}_{N} \in \R^{{N}\times {N}}$ of $\mathcal{L}$ such that
$$
\mathbf{L}_{N}\,a~=~b,
$$
where the $i$-th column $\mathbf{c}_i$ of $\mathbf{L}_{N}$ contains the coefficient of the expansion of $\mathcal{L}_{N}\psi_i$ in the basis. Indeed, by taking the $i$-th canonical vector $a = \mathbf{e}_i$, we obtain $\mathcal{L}_{N}\psi_i = (\mathbf{L}_{N}\mathbf{e}_i)^{\top}\Psi(x)=\mathbf{c}_i^{\top}\Psi.$
One can easily see that the eigenfunctions $\hat{\phi}_{\hat{\lambda}_i}$ of $\mathcal{L}_{N}$, associated with the eigenvalues $\hat{\lambda}_i$, are given by the right eigenvectors $v_i$ of $\mathbf{L}_{N}$, i.e. $\hat{\phi}_{\hat{\lambda}_i}(x) = v_i^T \Psi(x)$ (see  \cite{mauroy2020koopman} for more details). They can be used to approximate the true eigenfunctions $\phi_{\lambda_i}$ of the infinitesimal generator $\mathcal{L}$, provided that $\lambda_i \approx \hat{\lambda}_i$, and therefore to construct a candidate Lyapunov function of the form \eqref{lyap}. In practice, we will focus on the so-called principal eigenfunctions, that is, the approximated eigenfunctions $\hat{\phi}_{\hat{\lambda}_i}$ will be selected so that the associated eigenvalues $\hat{\lambda}_i$ are close to the eigenvalues of the Jacobian matrix $J_F(0)$. Since there is no guarantee that these eigenfunctions approximate well those of $\mathcal{L}$, a validity region should be computed for the candidate Lyapunov function (see Section \ref{validation-section}).\\

\begin{remark}
When the vector field $F$ is analytic and $\mathcal{F}_N$ is spanned by monomials, it is well-known that $\sigma(J_F(0)) \subset \sigma(\mathbf{L}_{N})$ and, for any $\lambda\in\sigma(J_F(0))$, $\hat{\phi}_{\lambda} = \Pi \phi_{\lambda}$ where $\Pi$ is the truncation projection operator. Moreover, when the eigenvalues of $J_F(0)$ are non-resonant, it follows from Poincaré linearization theorem \cite{gaspard1995spectral} that the eigenfunctions are analytic. This ensures that the eigenfunctions are approximated by their truncated Taylor series. In the disk of analyticity, approximated eigenfunctions converge (pointwise) to the exact ones in the limit of an infinite number of basis monomials, and therefore the candidate Lyapunov function also converges to an exact Lyapunov function. However, the convergence is not guaranteed beyond the disk of analyticity, which might restrict the size of the inner approximation of the region of attraction.
\end{remark}

\subsection{Polynomial approximations}\label{Polapprox-section}

The validation methods described in the next section are valid for polynomial Lyapunov functions. They also require a polynomial vector field or, alternatively, a polynomial approximation of it. In this work, we will rely on two different polynomial approximations, that we will use to obtain a polynomial proxy of the Lyapunov function and of the vector field.

\paragraph{Taylor approximation.}
A polynomial approximation of a function $f$ can be obtained with its Taylor expansion $P(x)$ up to some order $s\in\NN$ (assuming that $f$ is smooth enough). In this case, the approximation error is given by the Lagrange remainder, which implies that there exists a positive constant $c$ such that 
\begin{equation}\label{ci}
|f(x)-P(x)| \leq c\|x\|_2^{s+1}.
\end{equation}
If the approximation is used with the vector field, it is assumed that (an upper bound on) the constant $c$ is known. Such a constant can be obtained empirically in a possibly conservative way

\paragraph{Minimax polynomial approximation.}
This method seeks for the best polynomial approximation $P^* \in \mathcal{F}_P^d$ of a function $f$ over the set $\X$, i.e. 
\begin{equation}\label{Remez}
P^*  = \operatorname*{argmin}_{P\in\mathcal{F}_P^d} \max_{x \in \X}|f(x)-P(x)|.
\end{equation}
Computing the solution of the above minimax problem is quite challenging. Alternatively, one can consider its discrete relaxation over a finite set of $m$ points $\mathcal{X}=\{x_k\}_{k=1}^m \subset \X$, which yields a linear program that can easily be solved. In this paper, we use in particular the Remez-type algorithm presented in \cite{watson1975multiple}, which is initialized with the set $\mathcal{X}$ containing Chebyshev nodes.
The solution to the relaxed problem provides a lower bound $\bar{\varepsilon}=\max_{k}|f(x_k)-P^*(x_k)|$ but their might be some (randomly chosen) points $\Tilde{x}$ in $\X$ such that $|f(\Tilde{x})-P^*(\Tilde{x})|>\bar{\varepsilon}$. Then the algorithm proceeds by adding those points $\Tilde{x}$ to $\X$ and recomputes the approximation $P^*$ until a stopping criterion is satisfied. The error $\bar{\varepsilon}$ provides a lower bound on the true error $\varepsilon = \|f-P^*\|_{\infty}$ that may be tight provided that the final set $\mathcal{X}$ is large enough. If needed (i.e. in the case of vector field approximation), a safety margin can still be added in order to ensure that $\varepsilon$ is overestimated. By default, we set $1.5\bar{\varepsilon}$. Note that an error bound on the vector field approximation could also be obtained and validated through SMT methods such as dReal \cite{gao2013dreal}. However, for high degree approximations, this can be computationally expensive compared to the approach described above.\\

\noindent Polynomial approximations are used in two cases:
\begin{itemize}
    \item[(a)] \textbf{Non monomial basis functions.} When general basis functions are used (e.g. Gaussian radial basis functions), the (non polynomial) Lyapunov function computed through the Koopman approach is approximated with the minimax method so that a polynomial proxy is obtained and processed as a candidate Lyapunov function in the validation step. Note that the approximation error does not need to be computed here since the candidate is further validated.
    \item[(b)] \textbf{Non polynomial vector fields.} A polynomial (Taylor or minimax) approximation of general vector fields can be used (i) directly before computing the Lyapunov function, which allows to use monomial basis functions along with the truncation operator (option 1), or (ii) after the computation for the validation step only, i.e. to obtain a polynomial approximation of $\dot{V}=\nabla V^T F$ (option 2). Note that the approximation error has to be computed for the validation step.
\end{itemize}

\noindent The flowchart in Figure \ref{fig:flowchart} depicts the use of the polynomial approximations at different stages of the Lyapunov function construction.

\section{Validation of the Lyapunov function and estimation of stability regions} \label{validation-section}

\tikzstyle{block} = [rectangle, draw, fill=blue!8, 
        text width=7em, text centered, rounded corners, minimum height=1em]
\tikzstyle{line} = [draw, -latex,line width=0.2mm]
\begin{figure*}[t]
\begin{center}
\begin{tikzpicture}[scale = 0.9, transform shape,mylabel/.style={thin, draw=black, minimum width=0.5cm, minimum height=4.5cm,fill=white,font=\Large}]

\node [block] (F_nonpol) {Non-polynomial vector field};

\node [block,above=2.5cm of F_nonpol] (F_pol) {Polynomial vector field};

\begin{scope}[on background layer]
\node[rectangle,minimum width=3.3cm,minimum height=4.7cm,fill=gray!10] [fit = (F_nonpol) (F_pol) ,label=above:Vector field type] (bx4) {};
\end{scope}

\node [block, right=2cm of F_pol] (mon_1) {Monomial basis functions};

\node [block, below=0.7cm of mon_1] (other) {Other basis functions};

\node [block, right=2cm of F_nonpol] (mon_2) {Monomials basis functions};

\node [block, right=2cm of mon_1] (Truncation) {Truncation projection};

\begin{scope}[on background layer]
    \node[rectangle,minimum width=3.3cm,minimum height=4.7cm,fill=gray!10] [fit = (mon_1) (other) (mon_2),label={[align=center] Choice of a\\basis function $\{\psi_k\}_{k=1}^N$}] (bx) {};
\end{scope}

\node [block, right=2cm of other] (L2_1) {$L^2$ orthogonal projection};

\node [block, right=2cm of mon_2] (L2_2) {$L^2$ orthogonal projection};

\begin{scope}[on background layer]
\node[rectangle,minimum width=3.3cm,minimum height=4.7cm,fill=gray!10] [fit = (Truncation) (L2_1) (L2_2),label={[align=center] Choice of 
a \\projection operator $\Pi$}] (bxx) {};
\end{scope}

\node [block, right=2cm of Truncation] (V_pol) {Polynomial function $V$};

\node [block, right=2cm of L2_2] (V_nonpol) {Non-polynomial function $V$};

\begin{scope}[on background layer]
\node[rectangle,minimum width=3.3cm,minimum height=4.7cm,fill=gray!10] [fit = (V_pol) (V_nonpol),label={[align=center] Type of \\Lyapunov candidate $V$}] {};
\end{scope}

\node [rectangle,rounded corners,right=0.3cm of F_nonpol,minimum width=1cm] (opt2) {{\footnotesize \textsc{Option 2}}}; 

\path [line,-] (F_nonpol.east) -- (opt2);
\path [line] (F_pol) --  (mon_1);
\path [line] (F_pol.east) -- (other.west);
\path [line] (opt2) -- (other.west);
\path[line] (opt2) -- (mon_2);
\path [line] (mon_1) --  (Truncation);
\path [line] (other) --  (L2_1);
\path [line] (mon_2) --  (L2_2);
\path [line] (Truncation) --  (V_pol);
\path [line] (L2_1.east) --  (V_nonpol.west);
\path [line] (L2_2.east) --  (V_pol.west);

\node[rectangle,text width=6em, text centered, rounded corners, minimum height=4em,below = 0.7cm of F_pol] (empty) {{\footnotesize Taylor/minimax approx. (\textsc{Option 1})}}; 

\node[text width=7em, text centered, rounded corners, minimum height=1em,below = 0.7cm of V_pol] (empty2) {{\small minimax approximation}};

\path[line,dashed,-] (F_nonpol) -- (empty); 
\path[line,dashed] (empty) -- (F_pol);

\path[line,dashed,-] (V_nonpol) -- (empty2); 
\path[line,dashed] (empty2) -- (V_pol);

\end{tikzpicture}
\end{center}\vspace{0.4cm}
\caption{The flowchart shows how the Lyapunov candidate is constructed depending on the nature of the vector field and of the basis functions (solid lines). Polynomial Lyapunov functions are obtained only when the basis functions are monomials. Dashed lines represent the polynomial approximations (Taylor or minimax) that are applied to the vector field or to the Lyapunov function.}
\label{fig:flowchart}
\end{figure*}
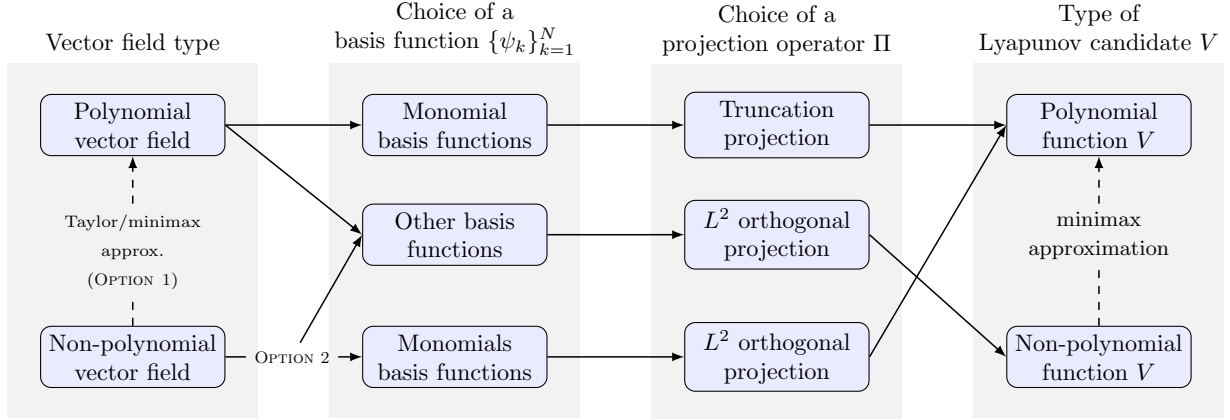

This validation step is the most critical part of the method. In this section, we will develop two validation techniques based on (i) SOS programming and (ii) a ``worst case'' approach combined with an adaptive grid.

The validation methods can be used with general equilibrium dynamics and are based on the polynomial approximations described above. We can rewrite (\ref{eq}) as
$$\dot{x}_i = P_i(x) + (F_i(x)-P_i(x)) \leq P_i(x) + \varepsilon_i(x),$$
for $i=1,\dots,n$, where $P_i(x)$ is a polynomial approximation of the vector field component $F_i$ and $\varepsilon_i(x)$ is a polynomial error term, i.e. $\varepsilon_i(x)=c_i \|x\|^{s+1}$ with $s$ assumed to be odd for the Taylor approximation, and $\varepsilon_i(x)=\varepsilon_i$ is a constant for the minimax approximation. Moreover, we denote by $V$ the candidate Lyapunov function which is assumed to be polynomial (possibly approximated through the minimax method when the basis functions are not polynomial). In order to obtain a rigorous stability certificate, we need to verify that
$$
\begin{array}{rcl}
\dot{V}(x) &=& \nabla V(x)^{\top} F(x)\\
&\leq& \nabla V(x)^{\top} P(x) + \sum_{i=1}^n\left|\dfrac{\partial V}{\partial x_i}\right|\varepsilon_i(x) < 0.
\end{array}$$
\noindent Note that we do not use the Cauchy-Schwartz inequality here, which might be more conservative and lead to polynomial inequalities of higher degree that could be unstable if used with SOS-type methods. The above inequality yields an inner approximation of the validity region \vspace{-0.1cm}
$$
\bar{\mathcal{S}} = \left\{x\in\X \mid (\nabla V^{\top}P)(x) + \displaystyle\sum_{j=1}^n\left|\dfrac{\partial V}{\partial x_j}\right|\varepsilon_j(x) < 0\right\}
$$ 
\noindent where $\bar{\mathcal{S}} \subset \mathcal{S}.$
The set $\bar{\mathcal{S}}$ is not described by a polynomial inequality, but can be rewritten as 
\begin{equation}\label{eq:validity_set}
\begin{split}
    \bar{\mathcal{S}} ~~&= \hspace{-0.6cm} \bigcap_{(r_1,\dots,r_n) \in \{0,1\}^n}\ \Bigl\{ x\in\X \mid \nabla V(x)^{\top} P(x)\\
    &\qquad+ \sum_{j=1}^n (-1)^{r_j}\dfrac{\partial V}{\partial x_j}\varepsilon_j(x) < 0 \Bigr\},\\
    & \triangleq \hspace{-0.6cm} \bigcap_{(r_1,\dots,r_n) \in \{0,1\}^n}\ \Bigl\{ x\in\X \mid R_{(r_1,\ldots,r_n)}(x) < 0 \Bigr\},
    \end{split}
\end{equation}
so that it is a semi-algebraic set.
In the case of the minimax approximation, the set $\bar{\mathcal{S}}$ is likely to be conservative near the origin, where the value $\nabla V(x)^{\top}P(x)$ is dominated by the error term $\varepsilon$. However, in contrast to the Taylor approximation, the minimax approximation is less conservative away from the origin, which is desirable for estimating large regions of attraction. According to Proposition \ref{multiple-lyap}, we can use both approximations to construct two Lyapunov candidates that focus on different regions of the state space (see the last example in Section \ref{Numerical-section}).

\subsection{Sum-of-squares programming}\label{SOS-section}

A region of attraction is obtained by solving the optimization problem \eqref{eq:ROA} with the approximate validity region $\bar{\mathcal{S}}$ instead of $\mathcal{S}$. It is well established that, for polynomial Lyapunov candidates and semi-algebraic sets, this problem can be solved by using SOS programming methods relying on Positivstellensatze relaxation (see e.g. \cite{meng2020application,parrilo2003semidefinite}). 

Considering the approximated validity set \eqref{eq:validity_set}, we can rewrite the optimization problem \eqref{eq:ROA} as 
\begin{equation}
\label{SOS-lyap}
\begin{array}{ll}
\displaystyle \max_{\gamma_1,\gamma_2 \geq 0} & \gamma_2-\gamma_1\\[0.3cm]
~~\textrm{s.t.} & \hspace{-0.3cm}\bigcdot~\sigma_1, \sigma_2\text{ is SOS}\\
& \hspace{-0.3cm}\bigcdot~  \forall \,(r_1,\dots,r_n) \in\{0,1\}^n, P \text{ is SOS, with }\\
& ~~ \hspace{-0.3cm} P = -R_{(r_1,\ldots,r_n)}-\sigma_1(V-\gamma_1) -\sigma_2(\gamma_2-V)\\
&
\end{array}
\end{equation}

\noindent where the constraints ensure that $\overline{\Omega_{\gamma_2} \setminus \Omega_{\gamma_1}} \subset \bar{\mathcal{S}}$. Note that the above SOS problem bears similarity to the problem formulated in \cite{meng2020application}. The main difference is that the Lyapunov function is built during the optimization process in that case, while it is given in advance and only validated with SOS programming in our case. 

Since the problem is bilinear in the variables $\gamma_1$, $\gamma_2$, $\sigma_1$ and $\sigma_2$, it is solved by using a bisection method over $\gamma_1$ and $\gamma_2$ and by solving a feasibility problem with SOS programming for each set of values $\gamma_1$ and $\gamma_2$. In practice, the maximal value $\gamma_2$ is chosen such that the set $\Omega_{\gamma_2}$ lies inside the region of interest $\X$. This prevents non-closed level-sets that could not be used to obtain a valid ROA.

\paragraph{Numerical limitations} 
As illustrated in Section \ref{Numerical-section}, SOS programming provides an efficient way to validate the Lyapunov function and compute large inner approximations of the ROA. However, the polynomials obtained with the Koopman operator method and subsequent approximation techniques might not be well-conditioned in general. This implies that the SOS method can be sensitive over the whole domain $\X$, especially for the choice of multipliers $\sigma_1$ and $\sigma_2$ in (\ref{SOS-lyap}). Moreover, interpretability issues can be encountered in the feasibility problem that is solved at each iteration of the bisection method. Finally, the SOS method is most commonly used with polynomials of total degree \mbox{$d \approx 5$} for low-dimensional systems ($n \leq 3$). Its use in higher-dimensional settings is notoriously prohibitive, which prevents SOS-based validation in that case (see Section \ref{sec:HD}).
 

\subsection{Adaptive grid}

We propose an additional validation method that can overcome some of the above-mentioned limitations of the SOS method, at the cost of reduced performance. This method relies on an adaptive grid, where the validation of each cell is based on a ``worst case'' approach. This technique is similar in spirit to the one described in \cite{bobiti2016sampling,siegelmann2023recurrence}. Although currently developed and implemented for polynomial basis functions, this method is conceptually more general and could be used to validate general Lyapunov functions obtained with non-polynomial basis functions. 

\paragraph{Estimation of the validity region.} We build an adaptive grid where each cell $\mathcal{C}$ is either validated if it can be shown that $\dot{V}(x) <0$ for all $x\in \mathcal{C}$ or is divided into $2^n$ smaller cells otherwise (unless a preset minimal size is reached). An approximation of the validity region is given by the union of all validated cells:
$$
\bar{\mathcal{S}} \supseteq \bigcup_{k\in K_{\textrm{val}}} \mathcal{C}_k
$$
where $K_{\textrm{val}}$ is the set of indices of validated cells.

In order to validate a cell $\mathcal{C} = [x_1,x_1+\delta] \times \cdots \times [x_n,x_n+\delta]$, we consider a worst-case approach. According to \eqref{eq:validity_set}, we first check that $R_{(r_1,\dots,r_n)}(x)<0$, for all $(r_1\dots,r_n) \in \{0,1\}^n$, for all $x \in v(\mathcal{C})$, where $v(\mathcal{C})$ denotes the set of vertices of $\mathcal{C}$. Then, it follows from the mean value theorem that $R_{(r_1,\dots,r_n)}(x)<0$ is satisfied for all $x\in \mathcal{C}$ if
\begin{equation}\label{criterion}
\begin{split}
  &\max_{x \in v(\mathcal{C})} \left(R_{(r_1,\dots,r_n)}(x)\right)\\ 
  &\qquad+ \frac{\sqrt{n} \delta }{2} \max_{x \in \mathcal{C}} \| \nabla R_{(r_1,\dots,r_n)}(x)\| < 0,
  \end{split}
\end{equation}
where $\sqrt{n} \delta/2$ is the maximal distance between the set $v(\mathcal{C})$ of vertices and any point $x\in \mathcal{C}$.

Since $R_{(r_1,\dots,r_n)}$ is a polynomial, the square of the norm of its gradient is also a polynomial (with coefficients $c_\alpha$) that can be bounded by
$$
\begin{array}{rcl}
\displaystyle\max_{x \in \mathcal{C}} \| \nabla R_{(r_1,\dots,r_n)}(x)\|^2 &=& \displaystyle\max_{x \in \mathcal{C}} \sum_{\alpha \in \mathbb{N}^n} c_\alpha x^\alpha \\[0.3cm]
&\leq& \displaystyle\sum_{\alpha \in \mathbb{N}^n}  \displaystyle\max_{x \in v^*(\mathcal{C})} \left(c_\alpha x^\alpha \right)
\end{array}
$$
where we used the triangular inequality and the fact that the extremal values of monomials is attained on the set $v^*(\mathcal{C})=v(\mathcal{C}) \cup \{x\in \mathcal{C}:\exists j \textrm{ s.t. }x_j=0\}.$

\paragraph{Estimation of the ROA.} It remains to compute the largest approximation of the ROA by solving \eqref{eq:ROA} with the validity region $\bar{\mathcal{S}}$ estimated above. Similarly to the previous SOS-based validation method, this is done through a bisection method. At each step, for a given pair of values $(\gamma_1,\gamma_2)$, a set of cells such that their union contains the associated sublevel sets is identified, i.e. one finds the smallest set of indices $K_{\textrm{level}}$ such that
$$
\overline{\Omega_{\gamma_2} \setminus \Omega_{\gamma_1}} \subset \bigcup_{k\in K_{\textrm{level}}} \mathcal{C}_k .
$$
It follows that, if $K_{\textrm{level}} \subseteq K_{\textrm{val}}$, the region $\overline{\Omega_{\gamma_2} \setminus \Omega_{\gamma_1}}$ is contained in $\bar{\mathcal{S}}$, and therefore validated. The set $K_{\textrm{level}}$ is obtained in two steps.\\

\noindent \textsc{First step.} One identifies the cells that (are likely to) contain the level sets $\partial \Omega_{\gamma}$, with $\gamma=\gamma_1$ or $\gamma=\gamma_2$. We start by selecting all the cells $\mathcal{C}$ such that
\begin{equation}\label{crossing-cell}
\min_{x\in v(\mathcal{C})} V(x)<\gamma \quad  \textrm{and} \quad \max_{x\in v(\mathcal{C})} V(x)>\gamma.
\end{equation}
Such cells are shown in blue in  Figure \ref{fig:illustration-grid}. Next, we inspect the neighbors of those cells. Although they do not satisfy the above condition, these neighbor cells can contain the level set $\partial \Omega_{\gamma}$ (see the green cell in Figure \ref{fig:illustration-grid}). This is verified by evaluating the values of $V$ on the boundary $\mathcal{B}_{\mathcal{C}}$ of every neighbor cell $\mathcal{C}$. For practical reasons, only a finite set of points $x_k\in \mathcal{B}_{\mathcal{C}}$ can be considered. If at least one point satisfies $|V(x_k)-\gamma|<\delta \max_{x \in \mathcal{C}}\|\nabla V(x) \|$, where $\delta = \max_{x \in \mathcal{B}_{\mathcal{C}}} \min_k \|x-x_k\|$ is the fill distance of the points in $\mathcal{B}_{\mathcal{C}}$, we assume that the cell can possibly intersect $\partial \Omega_\gamma$ and should therefore be included in the set $K_{\mathrm{val}}$.\\
\begin{figure}[t!]
\centering
\includegraphics[width=0.4\textwidth]{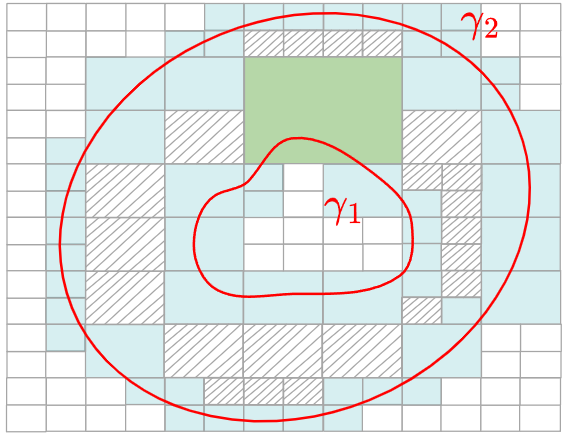}
\vspace{0.3cm}
\caption{Hypothetical example of an adaptive grid obtained through our proposed method. Blue cells contain the level set $\partial\Omega_{\gamma_1}$ or $\partial\Omega_{\gamma_2}$ since \eqref{crossing-cell} holds. In contrast, the green cell does not satisfy \eqref{crossing-cell}, although it intersects the level set $\partial\Omega_{\gamma_1}$. The hatched cells satisfy \eqref{hatched}.}
\label{fig:illustration-grid}
\end{figure} 

\noindent \textsc{Second step.} One selects the remaining cells $\mathcal{C} \subset \overline{\Omega_{\gamma_2} \setminus \Omega_{\gamma_1}}$, which satisfy the condition
\begin{equation}\label{hatched}
    \gamma_1 < V(x) < \gamma_2 \quad \forall x\in v(\mathcal{C}).
     \end{equation}
    These cells are hatched in Figure \ref{fig:illustration-grid}.

The proposed validation methodology is general. Provided that one can compute a bound on $\|\nabla V\|$ and $\| \nabla \dot{V}\|$ over a cell, the method could be directly extended to general basis functions, without polynomial approximation of vector fields and Lyapunov functions. Note also that, similarly to SOS techniques, the method suffers from the curse of dimensionality, being typically limited to low-dimensional systems (i.e. dimension less or equal to $3$). We leave those perspectives for future development.


\section{Numerical examples}\label{Numerical-section}

We now illustrate our method with several examples, considering systems with polynomial and non-polynomial vector fields, as well ashigh-dimensional systems. Even though the method can be used with general basis functions, we focus on two specific cases: 
\begin{enumerate}
    \item Monomial basis functions $\{x^{\alpha_i}, \alpha_i \in \mathbb{N}^n, |\alpha_i|\leq d\}_{i=1}^{N}$ up to total degree $d \in \NN$,\vspace{0.3cm}
    \item Gaussian radial basis functions $\{e^{-\eta^2\|x-c_i\|^2}\}_{i=1}^{N}$ for some parameter $\eta>0$ and centers $\{c_i\}_{i=1}^{N} \subset \mathcal{D}_c$, with $\mathcal{D}_c \subset \X$.
\end{enumerate}

\noindent In all examples, the Lyapunov function is computed and validated for a dynamic rescaled so that the region of interest $\X$ is reduced to the set $[-1,1]^n$. This rescaling is needed to enhance the performance of the SOS methods, which are implemented with the SOSTOOLS toolbox \cite{papachristodoulou2013sostools,prajna2002introducing} and Mosek solver. Note that the degrees of the SOS multipliers in \eqref{SOS-lyap} strongly affect the performance. Some care should be taken to tune them.

Experiments were run on an AMD Ryzen 7 5700U laptop with 8 cores and 16GB of RAM. They took from $5$ seconds (polynomial vector fieds with monomial basis functions) to $2$ minutes (non-polynomial vector field with minimax approximation), for a suitable choice of tolerance parameters (bisection method) and degree of multipliers (SOS method).

\subsection{Polynomial vector field} 
\label{exemple1}

\begin{figure*}[t!]
\centering
    \includegraphics[width=0.45\textwidth]{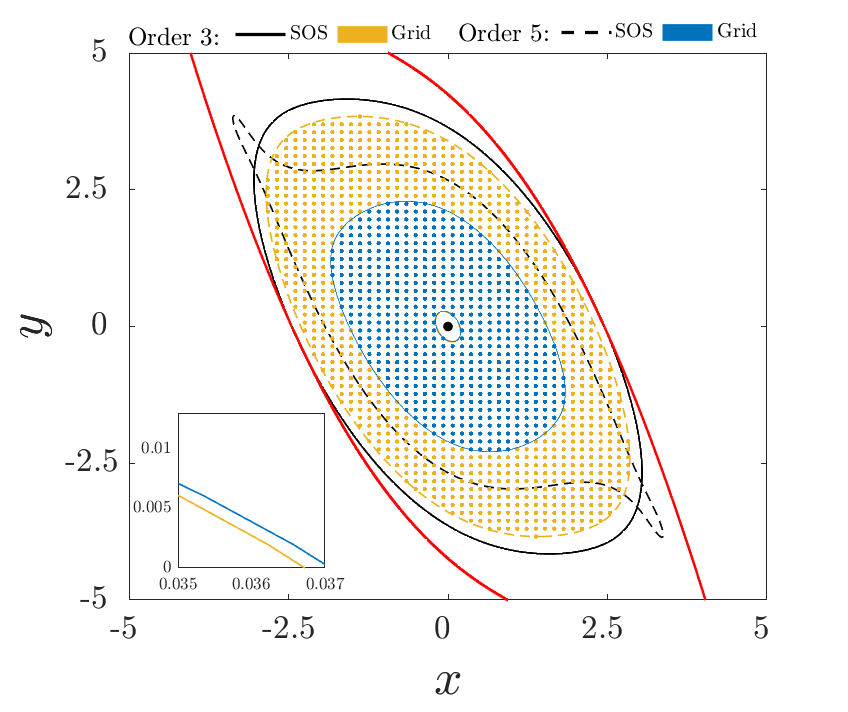}
    \includegraphics[width=0.45\textwidth]{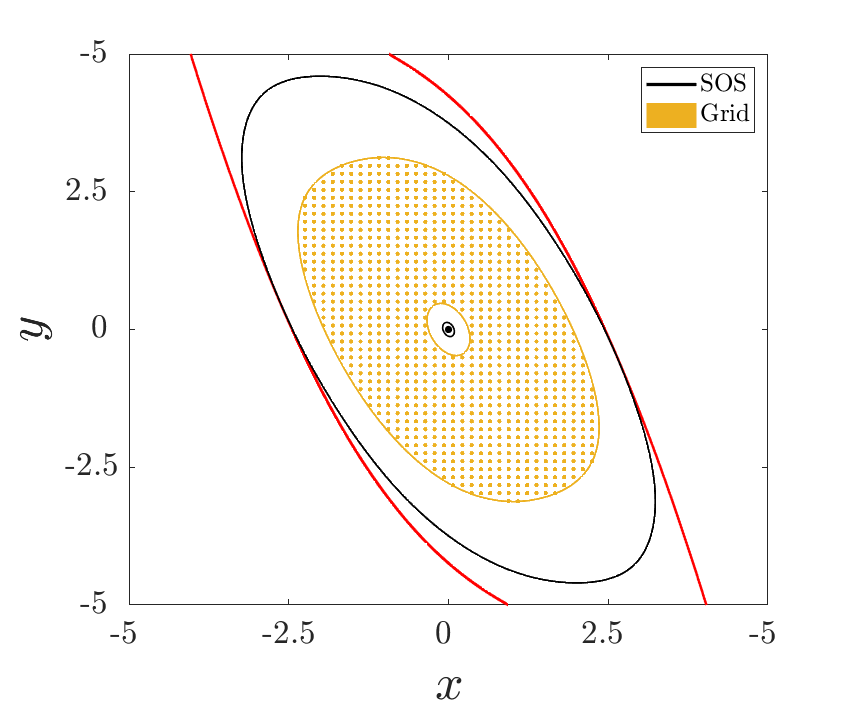}
\caption{Inner approximations of the ROA are obtained for the dynamics (\ref{sys1}). The boundary  of the ROA is in red. (Left) The approximation of the ROA is obtained with monomial basis functions up to degree 3 and 5. The approximation obtained with SOS-based validation is in black while the approximation obtained with the adaptive grid is in yellow (monomials up to degree 3) and blue (monomials up to degree 5). The inset shows the boundary of the estimated ROA near the equilibrium, when it is computed with the adaptive grid. (Right) The approximation of the ROA is obtained with 25 Gaussian radial basis functions with $\mathcal{D}_c = [-1,1]^2$ and $\eta = 0.9$. The approximation obtained with SOS-based validation (performed with a minimax polynomial approximation of degree 12) is in black while the approximation obtained with the adaptive grid is in yellow.}   

\label{fig:Pol-ex}
\end{figure*}

The first example focuses on the case of a polynomial vector field. \paragraph{Example 1} Consider the dynamics (see \cite{mauroy2014global})
\begin{equation}\label{sys1}
\left\{
\begin{array}{rcl}
\dot{x} & = & y  \\
\dot{y} & = & -2x-y+\dfrac{x^3}{3}
\end{array}\right.
\end{equation}
over $\X=[-5,5]^2$, which admit the origin as a stable equilibrium. The ROA is obtained by using a basis of monomials (Figure \ref{fig:Pol-ex}, left) and Gaussian radial basis functions (Figure \ref{fig:Pol-ex}, right). In the first case, monomials are used up to degree 3 and 5. We observe that the sublevel set extracted with the SOS method is consistent with the results shown in \cite{mauroy2014global} (without rigorous validation) and that $\gamma_1$ has been set to 0. The adaptive grid provides a smaller approximation which excludes a small neighborhood of the equilibrium ($\gamma_1 \neq 0$). For the two validation methods, the best approximation is obtained with monomials of degree 3. 
In the second case, 25 Gaussian radial basis functions are used and a minimax polynomial approximation of the Lyapunov function is obtained with polynomials up to degree 12. The approximation of the ROA provides slightly better results than in the case of monomials, when using SOS validation. It is noticeable that, in this case, the Lyapunov function (of degree 12) has almost the same number of nonzero terms as the one computed with monomials up to degree 5. This suggests that Gaussian radial basis functions are well suited for applications to higher-dimensional systems, where the use of monomials is intractable. The approximation obtained with the adaptive grid is not as good but yet similar to the one obtained with monomials up to degree 5.

\subsection{Non-polynomial vector fields}\label{non-poly VF}

\begin{figure*}[t!]
\centering
\includegraphics[width=0.45\textwidth]{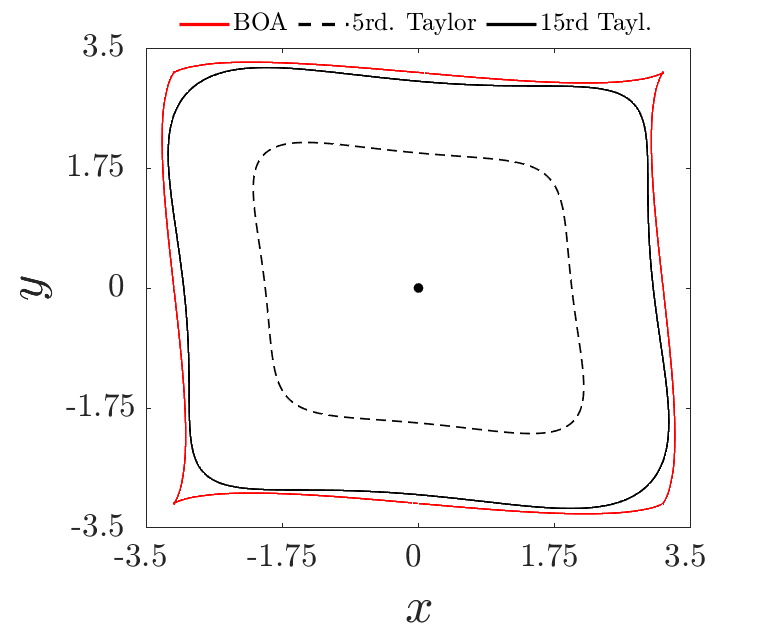}
\includegraphics[width=0.45\textwidth]{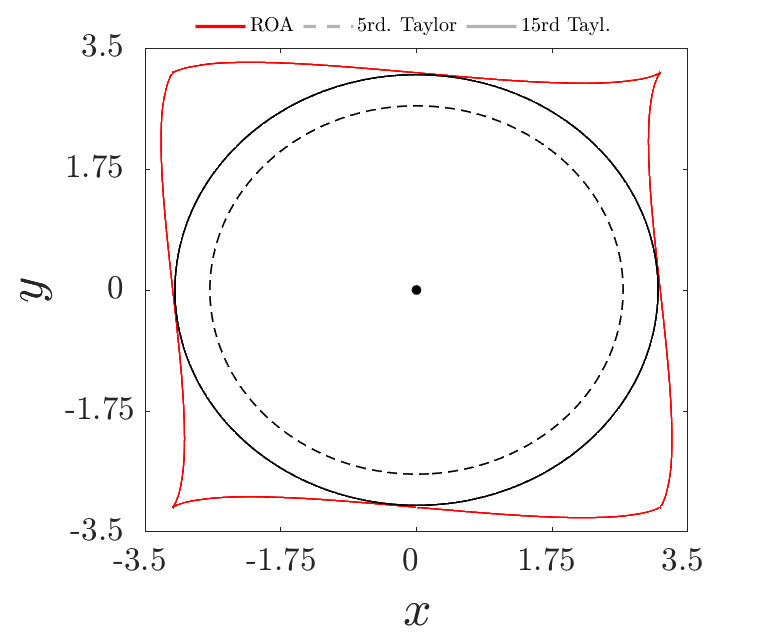}
\caption{Inner approximations of the ROA are obtained for the dynamics (\ref{sys2}). The boundary of the ROA is shown in red. (Left) The approximation of the ROA is obtained with monomial basis functions of degree 5, for a Taylor approximation of the vector field of order 5 (with $c_1=c_2 = 0.7)$) (dashed line) and of order 15 (with $c_1=c_2=2\times 10^{-4}$) (solid line). (Right) The approximation of the ROA is obtained with 9 Gaussian radial basis functions (with $\mathcal{D}_c = [-1,1]^2$ and $\eta = 0.1$), for a Taylor approximation of the vector field of order 5 and 15. The Lyapunov function is computed with the approximated vector field. The orthogonal projection is computed over the set $[-0.1,0.1]^2 \subset \X$.} 
\label{fig:NonPol-ex1}
\end{figure*}

Now, we show the performance of the method in the case of non-polynomials vector fields.
\paragraph{Example 2} Consider the dynamics (see \cite{mauroy2020koopman})
\begin{equation}\label{sys2}
\left\{
\begin{array}{rcl}
\dot{x} & = & K\sin(x-y)-\sin(x), \\
\dot{y} & = & K\sin(y-x)-\sin(y) 
\end{array}\right.\vspace{0.2cm}
\end{equation}
with $K = 0.2$ and $\X = [-3.5,3.5]^2$, which admits a stable origin. The vector field is analytic with an infinite radius of convergence and is naturally approximated through its Taylor expansion. The values $c_i$ appearing in the approximation error \eqref{ci} for each vector field component are obtained in a heuristic way by inspection over the set $\X$. Both monomial basis functions and Gaussian radial basis functions are used. As shown in Figure \ref{fig:NonPol-ex1}, better results are obtained with a higher order Taylor approximation, as expected. The best approximation is obtained with monomials provided that the order of the Taylor approximation is high enough. For low orders of the Taylor approximation, Gaussian basis functions provide better results.

\begin{figure}[t]
\centering
\includegraphics[width=0.43\textwidth]{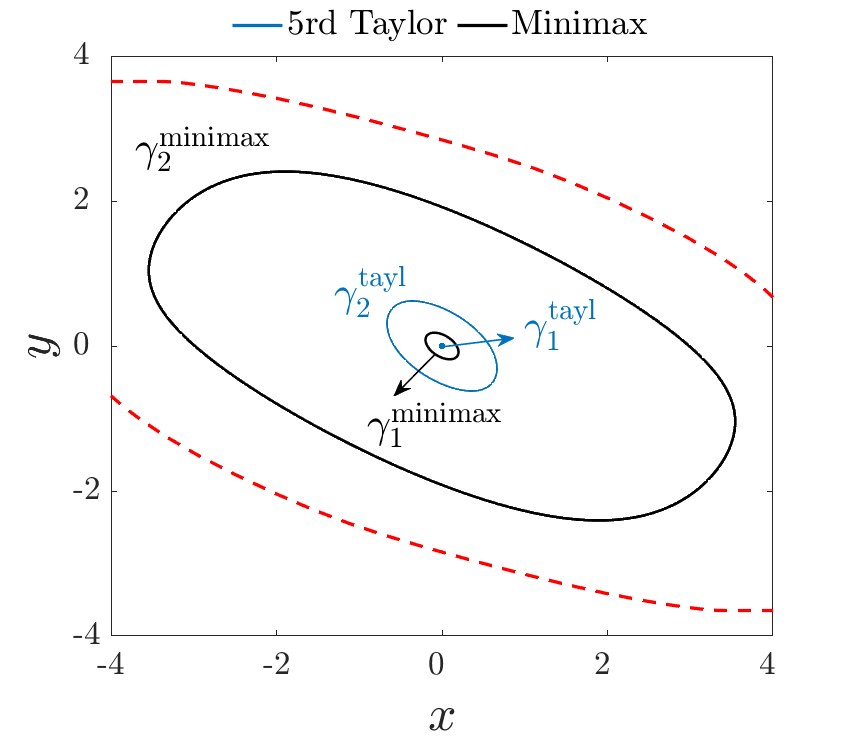}
\caption{Inner approximations of the ROA are obtained for the dynamics (\ref{sys3}). Black curves are the smallest and largest level sets of the Lyapunov function obtained with the minimax approximation of the vector field (polynomials up to degree 12, approximation error set to 0.028). Blue curves are the smallest and largest level sets of the Lyapunov function obtained with a 5th order Taylor expansion of the second component of the vector field (with $c_2=1.6\times10^3$). In both cases, basis monomials are used up to degree 5. By combining both results, the sublevel set $\Omega_{\gamma_{2}^{\mathrm{rem}}}$ provides an inner approximation of the ROA, and of the set of trajectories that remain in $\mathbb{X}$ (dashed red curves). }   
\label{fig:NonPol-ex2}
\end{figure}

\paragraph{Example 3\label{ex3}} 
Consider the dynamics (see \cite{papachristodoulou2005analysis})
\begin{equation}\label{sys3}
\left\{
\begin{array}{rcl}
\dot{x} & = & y, \\[0.3cm]
\dot{y} & = & \dfrac{-(x+y)}{\sqrt{1+(x+y)^2}}
\end{array}\right.\vspace{0.2cm}
\end{equation}
over the set $\X=[-4,4]^2$, which is characterized by a globally stable origin. The ROA of the origin is $\mathbb{R}^2$ but we aim at approximating the largest invariant set in $\mathbb{X}$. The Taylor expansion of the second component of the vector field has a small radius of convergence. 
Therefore, we approximate the vector field using the minimax approximation. The approximation of the ROA lies between the level sets $\gamma_1^{\mathrm{minimax}}$ and $\gamma_2^{\mathrm{minimax}}$(black curves in Figure \ref{fig:NonPol-ex2}). Since $\gamma_1^{\mathrm{minimax}}>0$, one cannot conclude that the trajectories converge to the origin. To overcome this issue, we also use a Taylor expansion of the vector field. In this case, we obtain an approximation of the ROA that lies between the level sets $\gamma_1^{\mathrm{tayl}}$ and $\gamma_2^{\mathrm{tayl}}$ (dashed black curves in Figure \ref{fig:NonPol-ex2}), with $\Omega_{\gamma_1^{\mathrm{tayl}}}=\{0\}$ and $\Omega_{\gamma_1^{\mathrm{minimax}}} \subset \Omega_{\gamma_2^{\mathrm{tayl}}}$. Then, according to Proposition \ref{multiple-lyap}, we can combine both results and conclude that $\Omega_{\gamma_{2}^{\mathrm{minimax}}}$ is an inner approximation of the ROA of the origin. As expected, this approximation tends to capture the set of trajectories that remain inside $\X$ (between the dashed red curves).

\subsection{Towards higher dimensions}\label{sec:HD}
In this section, we provide preliminary results of the application of our method to high-dimensional systems. To do so, we consider the replicator dynamics 
\begin{equation}\label{replicator}
\dot{x}_i = x_i((Ax)_i-x^{\top} A x),
\end{equation}
with $i\in\{1,\ldots,n\}, n\in\{4,5,6,8,10,12,14\}$, and $A\in\mathbb{R}^{n\times n}$. This system is forward invariant in (and thus reduces to) the $(n-1)$-dimensional simplex $S=\{x \in \mathbb{R}^n_+ : x_1+\cdots+x_n=1$\}.
The matrix $A$ is designed so that the system admits at least one equilibrium point at a vertex of the simplex.
Since both validation techniques presented in Section \ref{validation-section} are not tractable here, we assess the performance of the candidate Lyapunov function in an empirical way as follows.
Let $W \triangleq \{q_i\}_{i=1}^{K}$ be a set of sample points uniformly distributed in the region of attraction of the equilibrium. (To do so, sample points are randomly distributed inside the simplex $S$ and those which do not generate trajectories converging toward the equilibrium are discarded.) Next, we define the ratios
\begin{equation*}
    r_1=\frac{\#(\mathcal{S} \cap W)}{\# W}=\frac{\#\{q_i \in W : \dot{V}(q_i) < 0 \}}{K}
\end{equation*}
and 
\begin{equation*}
\begin{split}
    r_2 & =\max_{\gamma_1,\gamma_2}\frac{\#(\Omega_\gamma \cap W)}{\# W} \quad \textrm{s.t. } \overline{\Omega_{\gamma_2} \setminus \Omega_{\gamma_1}} \cap W \subset \mathcal{S},\\
    & =\max_{\gamma_1,\gamma_2}\frac{\#\{q_i \in W : \gamma_1 \leq V(q_i) \leq \gamma_2 \}}{K} \\
    & \qquad \textrm{s.t. } \gamma_1 \leq V(q_i)\leq\gamma_2 \Rightarrow \dot{V}(q_i) <0, \\
    \end{split}
\end{equation*}
where $\#$ denotes the cardinality of a set.

\begin{figure}[b!]
\centering
\includegraphics[width=0.45\textwidth]{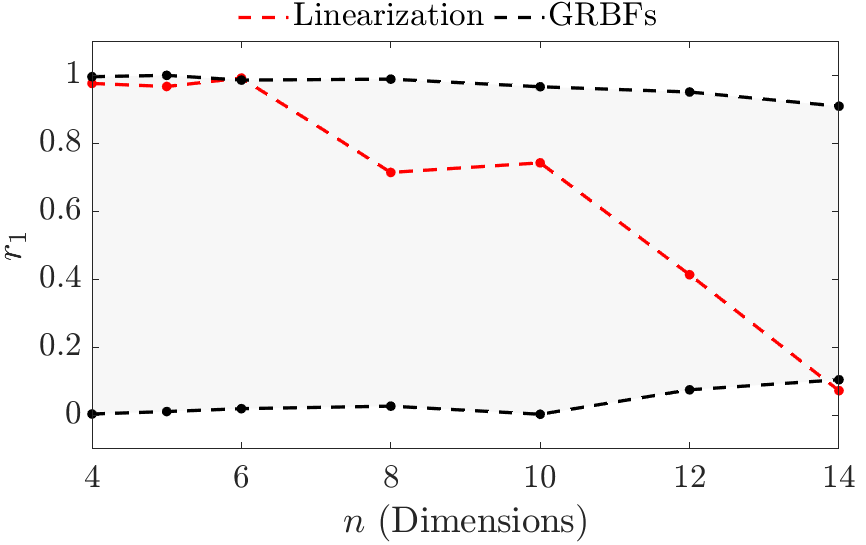}\vspace{0.4cm}
\includegraphics[width=0.45\textwidth]{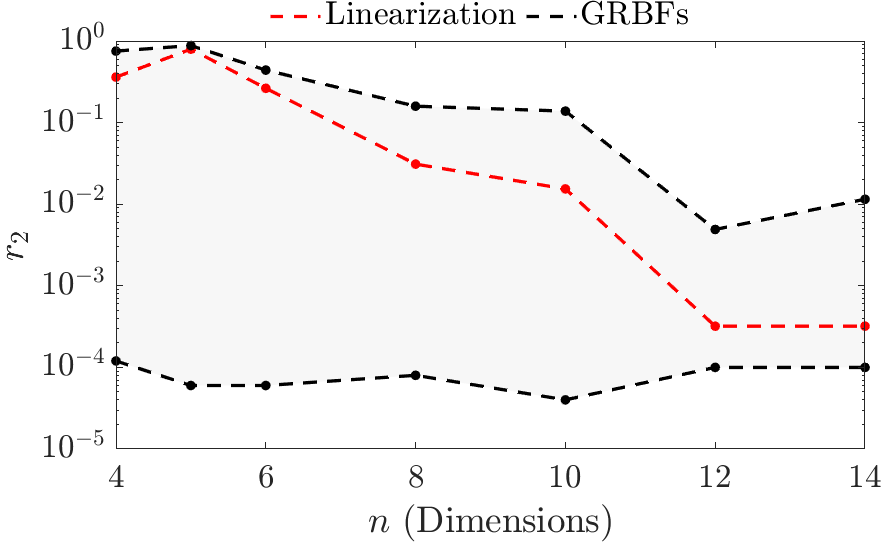}
\caption{Values of $r_1$ (above) and $r_2$ (below) are computed for the replicator dynamics \eqref{replicator} with different dimensions. Our Koopman operator based approach is used with $20$ GRBFs, with randomly distributed centers $c_i$ and parameter $\eta$ (extremal values in black). These results are compared with a quadratic Lyapunov function computed for the linearized dynamics (red curve).} 
\label{fig:HD}
\end{figure}

Clearly, a \textit{true} Lyapunov function should satisfy $r_1 = r_2 = 1$ for any sampling set $W$. The value $r_1$ corresponds to the classical property expected for a Lyapunov function and therefore provides a first glance at the quality of the approximation. The value $r_2$ is directly related to the size of the region of attraction estimated with a level set of the Lyapunov function. Note that the forward-invariance property of $S$ implies that the level sets of $V$ do not need to be closed. For every value $n\in\{4,5,6,8,10,12,14\}$, we have computed several candidate Lyapunov functions with 20 Gaussian radial basis function, where the parameter $\eta$ is sampled over $[0.01, 3]$ and the centers $c_i$ are randomly chosen in the simplexes $\{x\in\R^n | \sum_i x_i \leq a\}$ with $a$ varying in $[0.1, 1]$. Figure \ref{fig:HD} shows the minimum and maximum values of $r_1$ and $r_2$ (dotted black curves), computed with $K = 50000$ sample points. These results are compared with those obtained with the quadratic Lyapunov function $V(x) = x^{\top}Px$ related to the linearized dynamics, where $P  \succ 0$ is the solution of the Riccati equation \mbox{$J_F(0)^{\top}P + PJ_F(0) + I_n = 0$} and $I_n \in\R^{n\times n}$ is the identity matrix. 
For all dimensions, the Koopman operator-based approach shows good performance in terms of the value $r_1$, since a proper optimization over the centers $c_i$ and the parameter $\eta$ allows to reach a value $r_1 \approx 1$. However, the inner approximation of the region of attraction is conservative, in the sense that the method captures a small portion of the region of attraction, and the performance rapidly degrades as the dimension increases. Yet the proposed approach (with optimized parameters) outperforms the quadratic Lyapunov function for higher dimensions, while using only a fairly small number of basis functions. Note finally that these preliminary results do not necessarily extend to any general higher-dimensional systems, but open research perspectives in this direction.\\

\begin{remark}
SOS-based validation was also performed in the above example for a polynomial approximation of $V$ obtained through a Taylor expansion up to order $6$. The parameters $\eta$ and $\{c_i\}_{i}$ were selected as the ones maximizing the value of $r_2$ (see Figure \ref{fig:HD}). In the case $n=4$, the method provides reasonable approximation of $r_2 = 0.6$ ($r_2 = 0.75$ according to our empirical validation). For $n \geq 5$, SOS-based validation was intractable, a result which is consistent with the findings provided in \cite{deka2022koopman}. 
\end{remark}


\section{Conclusion and Perspective}
\label{sec:conclu}

In this paper, we developed a numerical method based on the Koopman operator framework to compute inner approximations of ROAs for general (i.e. non-polynomial) dynamics with general basis functions. Our method relies on polynomial approximation techniques, which are required to obtain semi-algebraic approximations of the validity sets of Lyapunov functions. In particular, we proposed two validation methods (based on SOS programming and on an adaptive grid) providing rigorous stability certificates. Our method was illustrated with different examples, for which good inner approximations of the ROA were obtained, especially with SOS validation techniques.

The present work opens up several perspectives for future research. The proposed method has been developed for general basis functions (e.g. radial basis functions), the use of which could be further explored. In the same line, the main principles of the validation technique based on the adaptive grid does not directly rely on polynomials, in contrast to SOS techniques, and could be further developed to depart from polynomial approximation. Polynomialization techniques \cite{gu2011qlmor,hemery2021compiling} could also be leveraged to circumvent an explicit computation of polynomial approximation error bounds. This approach could potentially allow the use of SOS-based validation techniques, but at the cost of higher-dimensional systems where such methods quickly become intractable. An alternative verification approach based on SMT techniques (e.g. dReal \cite{gao2013dreal}) could also be considered in this direction, but at a high computational cost. Moreover, although the current framework is limited to low-dimensional systems, preliminary results pave the way to its extension to higher dimensions. In this context, a proper selection of basis functions will play a crucial role and novel validation methods will also be needed. Toward this end, the framework developed in this paper could be extended to data-driven settings (e.g. via kernel-based approaches \cite{giesl2016approximation}), along with probabilistic methods (e.g. scenario-based approach \cite{calafiore2006scenario}). 

\bibliography{ref_paper_ROA}

\end{document}